\documentclass[reqno,12pt]{amsart}
 \usepackage{amssymb, latexsym, euscript}
 \newtheorem{theorem}{Theorem}[section]
 \newtheorem{proposition}[theorem]{Proposition}
 
 \newtheorem{lemma}[theorem]{Lemma}
 \newtheorem{definition}[theorem]{Definition}
 \newtheorem{corollary}[theorem]{Corollary}

\theoremstyle{definition}

\newtheorem{remark}[theorem]{Remark}

\numberwithin{equation}{section}

\newcommand{\cC}{\mathcal{C}} \newcommand{\cP}{\mathcal{P}} \newcommand{\cT}{\mathcal{T}}    \newcommand{\cD}{\mathcal{D}}     \newcommand{\sH}{\mathfrak{H}} \newcommand{\sL}{\mathfrak{L}}
\begin{document} \title[On dual definite subspaces in Krein space]{On dual definite subspaces in Krein space} \author[A.~Kamuda]{A.~Kamuda}  \author[S.~Kuzhel]{S.~Kuzhel} \author[V.~Sudilovskaya]{V.~Sudilovskaya}

\address{AGH University of Science and Technology \\ 30-059 Krak\'{o}w, Poland}
\email{kamudal@agh.edu.pl}

\address{AGH University of Science and Technology \\ 30-059 Krak\'{o}w, Poland}
\email{kuzhel@agh.edu.pl}
\address{Kyiv Vocational College\\ Kiev, Ukraina}
\email{veronica.sudi@gmail.com}


\maketitle
\begin{abstract}
Extensions of dual definite subspaces to dual maximal definite ones are described.  The concepts of 
dual quasi maximal subspaces and quasi basis are introduced and studied. The obtained results are applied to
the classification of $\cC$-symmetries.
\end{abstract}

\section{Introduction} 
Let $\mathfrak{H}$ be a Hilbert space with inner product $(\cdot,\cdot)$  linear in the first argument
and let $J$ be a non-trivial fundamental symmetry, i.e., $J=J^*$, $J^2=I$, and $J\not={\pm{I}}$.
The space $\mathfrak{H}$ endowed with the indefinite inner product
\begin{equation}\label{new1}
[f,g]=(J{f}, g)
\end{equation}
is called a Krein space $(\mathfrak{H}, [\cdot,\cdot])$.

A (closed) subspace $\mathfrak{L}$ of the
Hilbert space $\mathfrak{H}$ is called {\it nonnegative, positive, uniformly positive} with respect to the indefinite
inner product $[\cdot,\cdot]$ if, respectively,  $[f,f]\geq{0}$, \ $[f,f]>0$, \ $[f,f]\geq{\alpha}\|f\|^2, (\alpha>0)$ for all $f\in\mathfrak{L}\setminus\{0\}$.
Nonpositive, negative and uniformly negative subspaces are introduced similarly.

In each of the above mentioned classes we can define maximal subspaces.
For instance, a closed positive subspace $\mathfrak{L}$ is called \emph{maximal positive} if $\mathfrak{L}$
is not a proper subspace of a positive subspace in $\sH$. The concept of maximality for other classes is defined similarly.
A subspace $\mathfrak{L}$ of ${\mathfrak H}$ is called \emph{definite} if it is either positive or negative.
 The term \emph{uniformly definite} is defined accordingly.

Subspaces $\mathfrak{L}_\pm$  of  $\mathfrak{H}$  is called \emph{dual subspaces} if  $\mathfrak{L}_-$ is nonpositive,
$\mathfrak{L}_+$ is nonnegative,  and $\mathfrak{L}_\pm$ are orthogonal with respect to $[\cdot, \cdot]$,  that is $[f_+, f_-]=0$ for all
$f_{+}\in\mathfrak{L}_+$ and all $f_{-}\in\mathfrak{L}_-$.

The subject of the paper is  dual definite subspaces. 
Our attention is mainly focused on dual definite subspaces  $\mathfrak{L}_\pm$
with additional assumption of the density of their algebraic sum\footnote{The brackets in (\ref{e8}) indicates that
${\mathfrak L}_\pm$  are orthogonal with respect to $[\cdot, \cdot]$.} 
\begin{equation}\label{e8}
{\cD}={\mathfrak L}_+[\dot{+}]{\mathfrak L}_-
\end{equation}

At first glance, the density of $\cD$ in $\sH$ should imply the maximality of definite subspaces ${\mathfrak L}_\pm$
in the Krein space $(\sH, [\cdot,\cdot])$.
However, the results of \cite{R12}  show the existence of a densely defined sum \eqref{e8}
for which there are \emph{various extensions} to dual maximal definite subspaces  ${\mathfrak L}_\pm^{max}$: 
\begin{equation}\label{agga32b}
\cD={\mathfrak L}_+[\dot{+}]{\mathfrak L}_-  \ \rightarrow \  \cD_{max}={\mathfrak L}_+^{max}[\dot{+}]{\mathfrak L}_-^{max}.
\end{equation}

The decomposition \eqref{e8} is often appeared in the spectral theory of $\mathcal{PT}$-symmetric Hamiltonians
 \cite{AK_Bender}  as the result of closure of linear spans of positive and negative eigenfunctions
 and it is closely related to the concept of $\mathcal{C}$-symmetry in $\mathcal{PT}$-symmetric quantum mechanics (PTQM)
\cite{AK_Bender3, AK_Bender4}. The description of a symmetry $\mathcal{C}$ is one of the key points in PTQM
and it can be successfully implemented only in the case where the dual subspaces  in \eqref{e8} are \emph{maximal}.
This observation give rise to a natural question: how to describe all possible extensions of dual definite subspaces  ${\mathfrak L}_\pm$ to
 dual maximal definite subspaces  ${\mathfrak L}_\pm^{max}$?  
 
 In Section \ref{Sec2} this problem is solved with the
use of Krein's results on  non-densely defined Hermitian contractions \cite{ArTs, AK_Krein}.  The main result
(Theorem \ref{agga22b}) reduces the description of extensions \eqref{agga32b} to the solution of the operator 
equation \eqref{agga23}. 

Each pair of dual maximal definite subspaces ${\mathfrak L}_\pm^{max}$ generates an
associated Hilbert space $(\sH_G, (\cdot,\cdot)_G)$.  If ${\mathfrak L}_\pm^{max}$ are 
uniformly definite, then $\cD_{max}$  in \eqref{agga32b}
 coincides with $\sH$ and  $\sH=\sH_G$ (since the inner product $(\cdot,\cdot)_G$ is equivalent to the original 
 one $(\cdot,\cdot)$).  On the other hand, if     
${\mathfrak L}_\pm^{max}$ are only definite subspaces, then $\sH\not=\sH_G$  and the inner products
$(\cdot,\cdot)_G$,  $(\cdot, \cdot)$ are not equivalent.  In this case, the direct sum $\cD$ may lose the property
of being densely defined in the new Hilbert space $(\sH_G, (\cdot,\cdot)_G)$.   

We say that dual definite subspaces ${\mathfrak L}_\pm$  are \emph{quasi maximal}  if
there exists at least one extension \eqref{agga32b} such that  
 the set $\cD$ remains  dense in the new Hilbert space $(\sH_G, (\cdot,\cdot)_G)$ constructed by 
 $\cD_{max}$.

In Section \ref{4}, dual quasi maximal subspaces are characterized in terms of extremal extensions of symmetric
operators: Theorems \ref{agga28}, \ref{agga30}, Corollary \ref{agga31}. 
The theory of extremal extensions \cite{AK_Arlin, Arlin} allows one  to classify all possible cases:
$(A), (B), (C)$ (uniqueness/nonuniqueness 
Hilbert spaces $\sH_G$ which preserve the density of $\cD$).

Section \ref{5} deals with the operator of $\cC$-symmetry. 
Each pair of dual definite subspaces ${\mathfrak L}_\pm$ determines by \eqref{new5} an operator $\cC_0$ such that
$\cC_0^2=I$ and $J\cC_0$ is a positive symmetric operator in $\sH$. 
The operator  $\cC_0$  is called  \emph{an operator of $\cC$-symmetry}  if 
$J\cC_0$ is a self-adjoint operator in $\sH$. In this case, the notation $\cC$ is used instead of $\cC_0$.

Let $\cC_0$ be an operator associated with dual definite subspaces ${\mathfrak L}_\pm$. 
Its extension to the operator of $\cC$-symmetry $\cC$ is equivalent to the construction of 
 dual maximal definite subspaces ${\mathfrak L}_\pm^{max}$ in \eqref{agga32b}. 
This relationship allows one to use the classification $(A), (B), (C)$ in Section \ref{4}
for the solution of the following problems: (i) how many operators of $\cC$-symmetry 
can be constructed on the base of dual definite subspaces ${\mathfrak L}_\pm$? 
(ii) is it possible to define an operator of $\cC$-symmetry as the 
the extension by the continuity in the new Hilbert space  $(\sH_G,  (\cdot,\cdot)_G)$?

The concept of dual quasi maximal subspaces allows one to  introduce quasi bases in 
 Section \ref{6}. The characteristic properties of quasi bases are presented in Theorem \ref{agga38}
and Corollaries \ref{agga35} -- \ref{agga72}.   
The relevant examples are given. 

In what follows ${D}(H)$, $R(H)$ and $\ker{H}$ denote, respectively, the domain, the range,  and the kernel space of a linear operator $H$.
The symbol $H\upharpoonright_{\mathcal{D}}$ means the restriction of $H$ onto a set $\mathcal{D}$.
 Let $\sH$ be a complex Hilbert space.  Sometimes, it is useful to specify the inner product $(\cdot,\cdot)$ endowed with $\sH$. In that case the notation $(\mathfrak{H}, (\cdot,\cdot))$ will be used.

\section{Dual maximal subspaces}\label{Sec2}

\subsection{Extension of dual subspaces ${\mathfrak L}_\pm$ on to dual maximal subspaces ${\mathfrak L}_\pm^{max}$.}
 Let $(\mathfrak{H}, [\cdot,\cdot])$ be a Krein space with a fundamental symmetry $J$.
Denote
\begin{equation}\label{AK9}
\sH_+=\frac{1}{2}(I+J)\sH, \qquad \sH_-=\frac{1}{2}(I-J)\sH.
\end{equation}
The subspaces $\sH_{\pm}$  of
$\sH$ are orthogonal with respect to the initial inner product $(\cdot,\cdot)$ as well as with respect to the indefinite inner product $[\cdot,\cdot]$.
Moreover $\sH_{+}  \ (\sH_{-})$ is maximal uniformly positive (negative) with respect to $[\cdot,\cdot]$  and
\begin{equation}\label{AK10}
\sH=\sH_+[\oplus]\sH_-.
\end{equation}
The decomposition (\ref{AK10}) is called \emph{the fundamental decomposition} of the
Krein space  $(\mathfrak{H}, [\cdot,\cdot])$.

{\bf I.}  Each positive (negative) subspace ${\mathfrak L}_+$ (${\mathfrak L}_-$) in the Krein space $(\mathfrak{H}, [\cdot,\cdot])$
can be presented with respect to \eqref{AK10} as follows:
\begin{equation}\label{agga9}
\begin{array}{c}
{\mathfrak L}_+=\{x_++K_+^-x_+  :  x_+\in{M}_+\subseteq\sH_+\}, \\
  \sL_-=\{x_-+K_-^+x_-  :  x_-{\in}M_-\subseteq\sH_-\},
  \end{array}
\end{equation}
where $K_{+}^{-} : \sH_+\to\sH_-$  and  $K_{-}^{+} : \sH_-\to\sH_+$  are strong contractions\footnote{an operator $K$ is called strong contraction if $\|Kf\|<\|f\|$ for all nonzero  $f\in{D}(K)$}
 with the domains $D(K_{+}^{-})=M_+\subseteq\sH_+$  and $D(K_{-}^{+})=M_-\subseteq\sH_-$,  respectively.  Therefore,  the pair
of subspaces ${\mathfrak L}_\pm$ is uniquely determined by the
formula
\begin{equation}\label{agga20b}
{\mathfrak L}_\pm=(I+T_0)P_{\pm}D(T_0)
\end{equation}
where
\begin{equation}\label{neww4}
T_0=K_{+}^{-}P_++K_{-}^{+}P_-,  \quad D(T_0)=M_+\oplus{M_-},
\end{equation}
and $P_{+}=\frac{1}{2}(I+J)$ and $P_{-}=\frac{1}{2}(I-J)$ are orthogonal projection operators on $\sH_+$ and $\sH_-$, respectively.

By the construction,  $T_0$  is a strong contraction in $\sH$ such that 
\begin{equation}\label{fff1}
JT_0=-T_0J
\end{equation}
and its domain $D(T_0)=M_-\oplus{M_+}$ is a (closed) subspace in $\sH$.
The additional requirement of  duality of ${\mathfrak L}_\pm$ leads to
the symmetricity of $T_0$.  Precisely, the following statement holds.

\begin{lemma}\label{neww2}
The definite subspaces ${\mathfrak L}_\pm$ in \eqref{agga20b} are dual  if and only if  the operator  $T_0$ is a  symmetric strong contraction in $\sH$ 
and \eqref{fff1} holds.
\end{lemma}
\begin{proof} It sufficient to establish that the symmetricity of $T_0$ is equivalent to the orthogonality of
$\sL_\pm$  with respect to the indefinite inner product $[\cdot, \cdot]$.  Indeed,  for any
$x_{\pm}\in{M}_{\pm}$,
$$
0=[(I+T_0)x_+, (I+T_0)x_-]=((I-T_0)x_+, (I+T_0)x_-)=(x_+,T_0x_-)-(T_0x_+,x_-).
$$
Hence,  $(x_+, T_0x_-)=(T_0x_+, x_-)$,  $\forall{x}_{\pm}\in{M}_{\pm}$. The last equality
is equivalent to
$$
(T_0(x_++x_-), y_++y_-)=(T_0x_+, y_-)+(T_0x_-,y_+)=(x_++x_-, T_0(y_++y_-))
$$
for all $f=x_++x_-$ and $g=y_++y_-$  from the domain of $T_0$. Therefore, $T_0$ is a symmetric operator.
\end{proof}

The operator $T_0$  characterizes the `deviation'  of subspaces  ${\mathfrak L}_\pm$  with respect to  $\sH_{\pm}$
and it allows to characterize the additional properties of ${\mathfrak L}_\pm$.
\begin{lemma}[\cite{GKS}]\label{agga1}
Let ${\mathfrak L}_\pm$  be dual definite subspaces (i.e., the operator $T_0$ satisfies the condition of Lemma \ref{neww2}). Then
\begin{itemize}
\item[(i)]  the subspaces ${\mathfrak L}_\pm$ are uniformly definite $\iff$  $\|T_0\|<1$;
\item[(ii)]  the subspaces ${\mathfrak L}_\pm$ are definite but no uniformly definite  $\iff$  $\|T_0\|=1$;
\item[(iii)] the subspaces ${\mathfrak L}_\pm$ are maximal  $\iff$  $T_0$ is a self-adjoint operator defined on $\sH$.
\end{itemize}
\end{lemma}

By virtue of Lemmas \ref{neww2}, \ref{agga1},  the extension of dual definite subspaces ${\mathfrak L}_\pm$
on to dual maximal definite subspaces ${\mathfrak L}_\pm^{max}$  \emph{is equivalent to the extension of $T_0$ to
a self-adjoint strong contraction $T$ anticommuting with $J$}. In this case cf. \eqref{agga20b},
\begin{equation}\label{agga21}
{\mathfrak L}_\pm^{max}=(I+T)\sH_{\pm}.
\end{equation}

The next result is well known and it can be proved by various methods (see, e.g. ,  \cite{AK_Azizov}, \cite[Theorem 2.1]{PH}).
For the sake of completeness, principal stages of the proof based on the Phillips work  \cite{PH}  are given.

\begin{theorem}\label{agga2}
Let  ${\mathfrak L}_\pm$  be dual definite subspaces. Then there exist  dual  maximal definite subspaces  ${\mathfrak L}_\pm^{max}$ such that
${\mathfrak L}_\pm\subseteq{\mathfrak L}_\pm^{max}$.
\end{theorem}
\begin{proof}  
 For the construction of ${\mathfrak L}_\pm^{max}$ we should
prove the existence of a strong self-adjoint contraction $T$  which extends $T_0$ and anticommutes with $J$.
The existence of  a self-adjoint contraction extension  $T'\supset{T_0}$ is  well known \cite{AK_Krein, AK_Akhiezer}.
  However, we cannot state that  $T'$ anticommutes with $J$.
To overcome this inconvenience we modify $T'$ as follows:
\begin{equation}\label{fifa1}
T=\frac{1}{2}(T'-JT'J).
\end{equation}
It is easy to see that  $T$  is a required self-adjoint strong contraction 
because  $T$ anticommutes with $J$  and $T$  is an extension of $T_0$ (since $JT_0=-T_0J$). Therefore, the  dual  maximal subspaces ${\mathfrak L}_\pm^{max}\supseteq{\mathfrak L}_\pm$
can be defined by \eqref{agga21}.
\end{proof}

The set of all self-adjoint contractive extensions of $T_0$  forms an operator interval  $[T_\mu, T_M]$ \cite[Theorem 3]{AK_Krein}.
The end points of this interval: $T_\mu$ and $T_M$ are called the \emph{hard} and the \emph{soft} extensions of $T_0$,  respectively.
 \begin{corollary}\label{agga3}
 Let ${\mathfrak L}_\pm$ be dual definite subspaces.
  Then their extension to the dual maximal subspaces  ${\mathfrak L}_\pm^{max}$
 can be defined by \eqref{agga21} with
$$
T=\frac{1}{2}(T_\mu+T_M).
$$
 \end{corollary}
 \begin{proof}
 Let us proof that
 \begin{equation}\label{agga4}
 JT_\mu=-T_{M}J.
 \end{equation}

By virtue of  \cite[2.§108 p.380]{AK_Akhiezer},  the operators $T_{\mu}$ and $T_M$ can be defined:
\begin{equation}\label{kaa1}
T_{\mu}=T-\sqrt{I+T}Q_1\sqrt{I+T}, \quad T_{M}=T+\sqrt{I-T}Q_2\sqrt{I-T}
\end{equation}
where $T$ is a self-adjoint contractive extension of $T_0$ anticommuting with $J$ (its existence was proved in Theorem \ref{agga2}),
  $Q_1$ and $Q_2$ are the orthogonal projections onto the orthogonal complements of the manifolds $\sqrt{I+T}D(T_0)$ and $\sqrt{I-T}D(T_0)$ respectively.

It is obvious that $J(I+T)=(I-T)J$ (because  $T$ anticommutes with $J$). Furthermore, since $(I\pm T)$ are self-adjoint and positive, there exists unique square roots  operators $\sqrt{I\pm T}$.
 Let us consider the operator $S=J\sqrt{I+T}J$,  and compute:
$$
S^2=J\sqrt{I+T}J^2\sqrt{I+T}J=J(I+T)J=(I-T)J^2=I-T.
$$
Hence,  $S=\sqrt{I-T}$ or
$J\sqrt{I+T}=\sqrt{I-T}J.$ The latter relation yields that the unitary operator $J$ transforms
   the decomposition  $\mathfrak{H}=\sqrt{I+T}D(T_0)\oplus{Q}_1\mathfrak{H}$  in
$\mathfrak{H}=\sqrt{I-T}D(T_0)\oplus{Q_2\mathfrak{H}}$. This means that
 $JQ_1=Q_2J$.  The	above analysis and (\ref{kaa1})  justifies (\ref{agga4}).

 Due to the proof of Theorem \ref{agga2},  for the construction of $T$ in \eqref{fifa1}
 we can use  arbitrary self-adjoint contraction $T'\supset{T_0}$.
 In particular, choosing $T'=T_{\mu}$ and using \eqref{agga4}, we complete the proof.
\end{proof}

In general, the extension of dual definite subspaces ${\mathfrak L}_\pm$ on to dual maximal definite subspaces ${\mathfrak L}_\pm^{max}$
is not determined uniquely. To describe all possible cases we use the formula \cite{AK_Arlin, AK_Krein}
\begin{equation}\label{agga22}
T=T_\mu+(T_M-T_\mu)^{\frac{1}{2}}X(T_M-T_\mu)^{\frac{1}{2}}
\end{equation}
which gives a one-to-one correspondence between all self-adjoint contractive extensions $T$ of $T_0$ and
all nonnegative self-adjoint contractions $X$ in the subspace $\mathfrak{M}=\overline{R(T_M-T_\mu)}$.
\begin{theorem}\label{agga22b}
The self-adjoint contractive extension $T\supset{T_0}$  determines dual maximal subspaces ${\mathfrak L}_\pm^{max}$ in
\eqref{agga21} if and only if  the corresponding nonnegative self-adjoint contraction $X$ describing $T$ in \eqref{agga22}
is the solution of the following operator equation in $\mathfrak{M}$:
\begin{equation}\label{agga23}
X=J(I-X)J
\end{equation}
\end{theorem}
\begin{proof}  It follows from \eqref{agga4} that the subspace $\mathfrak{M}$ reduces $J$. Furthermore
$J(T_M-T_\mu)^{\frac{1}{2}}=(T_M-T_\mu)^{\frac{1}{2}}J$.  Taking these relations into account, we
conclude that the self-adjoint contraction $T$ in \eqref{agga22} anticommutes with $J$ if and only if
$X$ satisfies \eqref{agga23}
\end{proof}
\begin{remark}\label{agga32}
The equation \eqref{agga23} has an elementary solution $X=\frac{1}{2}I$  which corresponds to
the operator $T$ defined in Corollary \ref{agga3}.

Let $X_0\not=\frac{1}{2}I$ be a solution of \eqref{agga23}. Then the nonnegative self-adjoint contraction
$X_1=I-X_0$  is  also a solution of \eqref{agga23}.  Moreover, each self-adjoint nonnegative contraction
$X_\alpha=(1-\alpha)X_0+\alpha{X}_1, \ \alpha\in[0,1]$ is the solution of  \eqref{agga23}
Therefore, either the dual maximal definite subspaces ${\mathfrak L}_\pm^{max}\supset{\mathfrak L}_\pm$ are determined uniquely
or there are infinitely many such extensions.
\end{remark}

{\bf II.} The above results as well as the results in the sequel it is useful to rewrite with the help of the Cayley transform
of $T_0$:
\begin{equation}\label{agga7}
G_0=(I-T_0)(I+T_0)^{-1},     \qquad  T_0=(I-G_0)(I+G_0)^{-1}.
\end{equation}

In what follows \emph{we assume that the  direct sum \eqref{e8} of ${\mathfrak L}_\pm$ is a dense set in $\sH$}.
Then, 
the operator $G_0$ is a closed densely defined positive symmetric operator in $\sH$
with
$$
D(G_0)=\cD, \qquad  \ker(I+G_0^*)=\sH\ominus(M_-\oplus{M_+})
$$
and such that
\begin{equation}\label{AK71}
JG_0f=G_0^{-1}Jf, \qquad \forall{f}\in{D}(G_0)={\mathcal D}.
\end{equation}

\begin{remark}
 Every nonnegative self-adjoint extension
$G\supset{G_0}$ (i.e. $(Gf,f)\geq{0}$) is also a positive extension of $G_0$ (i.e. $(Gf,f)>{0}$ for $f\not=0$).
Indeed, if $(Gf,f)={0}$, then $Gf=0$ and $(Gf,g)=(f, G_0g)=0$ for all $g\in{D}(G_0)$. Therefore,
$f\perp{R}(G_0)=D(G_0^{-1})$ and $f=0$ since $D(G_0^{-1})$ is a dense set in $\sH$ by virtue of \eqref{AK71}.
\end{remark}

Self-adjoint positive extensions $G$ of $G_0$ are in one-to-one correspondence with the set of
contractive self-adjoint  extensions of $T_0$:
\begin{equation}\label{agga16}
T=(I-G)(I+G)^{-1},  \qquad G=(I-T)(I+T)^{-1}.
\end{equation}
In particular the Friedrichs extension $G_{\mu}$ of $G_0$ corresponds to the operator $T_\mu$, while
the Krein-von Neumann extension $G_M$ is the Cayley transform of $T_M$. The relation \eqref{agga4}
between $T_\mu$ and $T_M$ is rewritten as follows \cite[Theorem 4.3]{GKS}:
\begin{equation}\label{agga29}
JG_\mu=G^{-1}_MJ.
\end{equation}

It follows from \eqref{agga7} and Lemmas \ref{neww2},  \ref{agga1} (see also \cite[Proposition 4.2]{KS}) that
\emph{the dual maximal definite subspaces ${\mathfrak L}_\pm^{max}\supseteq{\mathfrak L}_\pm$  is in
one-to-one correspondence with positive self-adjoint extensions $G$ of $G_0$ satisfying the
additional condition} (cf. \eqref{AK71}):
\begin{equation}\label{agga14}
JGf=G^{-1}Jf, \qquad \forall{f}\in{D}(G)={\mathfrak L}_+^{max}[\dot{+}]{\mathfrak L}_-^{max}.
\end{equation}

\section{Krein spaces associated with dual maximal subspaces}
\subsection{The case of maximal uniformly definite subspaces.}\label{ref2}
Let ${\mathfrak L}_\pm^{max}$  be dual maximal uniformly definite subspaces.
Then:
\begin{equation}\label{e8b}
{\sH}={\mathfrak L}_+^{max}[\dot{+}]{\mathfrak L}_-^{max}.
\end{equation}

Relation \eqref{e8b} illustrates the variety of possible decompositions of the Krein space $(\sH, [\cdot,\cdot])$ onto its
maximal uniformly positive/negative subspaces. This property is characteristic for a Krein space and, sometimes,  it is used
for its definition \cite{AK_Azizov}.

 With decomposition \eqref{e8b}  one can associate a new inner product in $\sH$:
\begin{equation}\label{agga15}
(f, g)_G=[f_+, g_+]-[f_-, g_-], \qquad  f,g\in\sH
\end{equation}
($ f=f_++f_-, \ g=g_++g_-, \  f_\pm, \ g_\pm \in {\mathfrak L}_\pm^{max}$). By virtue of \eqref{agga21},
the relations $f_\pm=(I+T)x_\pm$, \ $g_\pm=(I+T)y_\pm$, \ $x_\pm, y_\pm \in \sH_\pm$  hold.
Taking \eqref{agga16} into account we rewrite \eqref{agga15} as follows:
\begin{equation}\label{fff7}
\begin{array}{l}
(f, g)_G=((I-T)x_+, (I+T)y_+)+((I-T)x_-, (I+T)y_-)= \\
((I-T)(x_++x_-), (I+T)(y_++y_-))=(Gf,g).
\end{array}
\end{equation}
Here $G$ is a bounded\footnote{`bounded'  since $\|T\|<1$ see Lemma \ref{agga1}} positive self-adjoint operator with $0\in\rho(G)$.
Therefore, the  subspaces ${\mathfrak L}_\pm^{max}$ determine the
new inner product
$$
(\cdot,\cdot)_1=(G\cdot, \cdot)=(\cdot,\cdot)_G,
$$
which is equivalent to the initial one $(\cdot,\cdot)$. 
The subspaces ${\mathfrak L}_\pm^{max}$ are mutually orthogonal with respect to $(\cdot,\cdot)_{G}$
in the Hilbert space $(\sH, (\cdot,\cdot)_G)$.

Summing up:  \emph{the choice of various dual maximal uniformly definite subspaces  ${\mathfrak L}_\pm^{max}$
generates infinitely many equivalent inner products $(\cdot,\cdot)_{G}$ of the Hilbert space $\sH$ but it
does not change the initial Krein space $(\sH, [\cdot,\cdot])$.}

\subsection{The case of maximal definite subspaces.}\label{ref1}
Assume that  ${\mathfrak L}_\pm^{max}$  are dual maximal definite subspaces but they are not uniformly definite.
Then the direct sum
\begin{equation}\label{agga18}
\cD_{max}={\mathfrak L}_+^{max}[\dot{+}]{\mathfrak L}_-^{max}
\end{equation}
is a dense set in the Hilbert space $(\sH, (\cdot, \cdot))$. The corresponding positive self-adjoint operator $G$ is unbounded.

Similarly to the previous case, with each of direct sum \eqref{agga18},
one can associate a new inner product $(\cdot,\cdot)_G=(G\cdot,\cdot)$ defined on $D(G)=\cD_{max}$ by the formula
\eqref{agga15}. The inner product $(\cdot,\cdot)_G$ is not equivalent to the initial one and
the linear space $\cD_{max}$ endowed with $(\cdot,\cdot)_{G}$ is a pre-Hilbert space.

Let $\sH_{G}$ be the completion of $\cD_{max}$ with respect to $(\cdot,\cdot)_{G}$.  The Hilbert space $\sH_{G}$ does not coincide with $\sH$.
The dual subspaces ${\mathfrak L}_\pm^{max}$ are orthogonal with respect to $(\cdot,\cdot)_{G}$ and, by construction,
the new Hilbert space $(\sH_{G}, (\cdot,\cdot)_{G})$ can be decomposed as follows:
\begin{equation}\label{e8c}
\sH_{G}=\hat{\sL}_+^{max}\oplus_{G}\hat{\sL}_-^{max},
\end{equation}
 where $\hat{\sL}_{\pm}^{max}$ are the completion of $\sL_\pm^{max}$ with respect to $(\cdot,\cdot)_{G}$.

The decomposition (\ref{e8c}) can be considered as the fundamental decomposition of the new
Krein space $(\sH_{G}, [\cdot, \cdot]_G)$ with the indefinite inner product
\begin{equation}\label{agga19}
[f, g]_G=({J_G}f, g)_G=(f_{+}, g_{+})_{G}-(f_{-}, g_{-})_{G},
\end{equation}
where $f=f_{+}+f_{-}, \ g=g_{+}+g_{-}, \ f_{\pm}, g_{\pm} \in \hat{\sL}_{\pm}^{max}$ and
${J_G}f=f_{+}-f_{-}$  is the fundamental symmetry in $\sH_G$.

Let $\mathfrak{D}[G]$ be the \emph{energetic linear manifold} constructed by the positive self-adjoint operator $G$. In other words,
$\mathfrak{D}[G]$ denotes the completion of $D(G)=\cD_{max}$ with respect to the  energetic norm
$$
\|f\|^2_{en}=\|f\|^2+\|f\|^2_G=\|f\|^2+(Gf, f).
$$
The set of elements $\mathfrak{D}[G]$ coincides with $D(\sqrt{G})$ and the 
energetic linear manifold is a Hilbert space $(\mathfrak{D}[G], (\cdot,\cdot)_{en})$
with respect to the energetic inner product
$$
(f, g)_{en}=(f,g)+(\sqrt{G}f,\sqrt{G}g), \qquad f,g\in{D(\sqrt{G})}=\mathfrak{D}[G]. 
$$

Comparing the definitions of $\sH_G$ and $\mathfrak{D}[G]$ leads to the conclusion that  the energetic linear manifold  $\mathfrak{D}[G]$ coincides
with the common part of $\sH$ and $\sH_G$,  i.e.,   $\mathfrak{D}[G]=\sH\cap\sH_G$.

\begin{lemma}\label{agga20}
The indefinite inner products $[\cdot, \cdot]$  and $[\cdot, \cdot]_G$ coincide on $\mathfrak{D}[G]$.
\end{lemma}
\begin{proof}
Indeed, taking \eqref{agga15} and \eqref{agga19} into account,
$$
[f,g]_G=[f_+,g_+]+[f_-,g_-]=[f_++f_-, g_++g_-]=[f, g], \quad \forall f,g\in{D(G)}.
$$
The obtained relation can be extended onto $\mathfrak{D}[G]$ by the continuity because $|[f,g]|\leq(f,g)$ and
 $|[f,g]|\leq(f,g)_G$.
\end{proof}

Summing up:  \emph{the choice of various dual maximal definite subspaces  ${\mathfrak L}_\pm^{max}$
generates infinitely many Krein spaces $(\sH_G, [\cdot,\cdot]_G)$  such that
the indefinite inner product  $[\cdot,\cdot]_G$ coincide with the original indefinite inner product $[\cdot,\cdot]$
on the energetic linear manifold $\mathfrak{D}[G]$.  The inner products $(\cdot,\cdot)$ and  $(\cdot,\cdot)_{G}$ restricted
on $\mathfrak{D}[G]$ are not equivalent.}

\section{Dual quasi maximal subspaces}\label{4}
\subsection{Definition and principal results.}\label{4.1}
 Let  ${\mathfrak L}_\pm$  be dual definite subspaces such that their direct sum \eqref{e8} is a dense set in 
 $\sH$ and let $G_0$ be the corresponding symmetric operator.
 Each positive self-adjoint extension $G$ of $G_0$ with additional condition \eqref{agga14} determines
 the Hilbert space\footnote{The Hilbert space $\sH_G$ coincides with $\sH$ if $G$ is a bounded operator} $(\sH_G, (\cdot,\cdot)_G)$.

 \begin{definition}\label{agga25}
The dual definite subspaces ${\mathfrak L}_\pm$ are called quasi maximal if
there exists positive self-adjoint extension $G\supset{G_0}$ with the condition \eqref{agga14}  and such that
 the domain $D(G_0)$ remains  dense in the new Hilbert space $(\sH_G, (\cdot,\cdot)_G)$.
\end{definition}

Obviously, each maximal definite subspaces are quasi maximal. 
For dual uniformly definite subspaces, the concept of  quasi-maximality is equivalent to maximality, i.e., 
 each quasi maximal uniformly definite subspaces have to be maximal uniformly definite. 

In general case of definite subspaces, the closure of dual quasi maximal subspaces
${\mathfrak L}_\pm$ with respect to $(\cdot,\cdot)_G$  coincides with subspaces $\hat{\sL}_\pm^{max}$
in the fundamental decomposition \eqref{e8c}, i.e., the closure of ${\mathfrak L}_\pm$ in $(\sH_G, (\cdot,\cdot)_G)$
gives dual maximal uniformly definite subspaces $\hat{\sL}_\pm^{max}$ of the new Krein space
$(\sH_G, [\cdot,\cdot]_G)$.

It is natural to suppose that the quasi maximality
can be characterized in terms of the corresponding
 positive self-adjoint extensions $G$ of $G_0$. For this reason, we recall \cite{Arlin, AK_Arlin}
 that a nonnegative self-adjoint extension  $G$ of
$G_0$ is called \emph{an extremal extension} if
\begin{equation}\label{bebe86}
\inf_{f\in{D}(G_0)}{(G(\phi-f),(\phi-f))}=0, \quad \mbox{for all} \quad \phi\in{D}(G).
\end{equation}
The Friedrichs extension $G_\mu$  and  the Krein-von Neumann extension $G_M$  are examples of extremal extensions of $G_0$.

\begin{theorem}\label{agga28}
Dual definite subspaces ${\mathfrak L}_\pm$ are quasi maximal if and only if
there exists an extremal extension $G$ of $G_0$ which satisfies \eqref{agga14}
\end{theorem}
\begin{proof}
Since $\|\phi-f\|^2_{G}=(G(\phi-f),(\phi-f))$, the condition (\ref{bebe86})
means that each element $\phi\in{D}(G)={\mathfrak L}_+^{max}[\dot{+}]{\mathfrak L}_-^{max}$ can be approximated by elements
$f\in{D}(G_0)={\mathfrak L}_+[\dot{+}]{\mathfrak L}_-$  in the Hilbert space $(\sH_{G}, (\cdot,\cdot)_G)$ 
if and only if $G$ is an extremal extension of $G_0$.
\end{proof}

\begin{remark}\label{new61}
In general, an extremal extension $G$ of $G_0$ is not determined uniquely.
Let $G_i,$ $i=1,2$ be extremal extensions of $G_0$ that satisfy \eqref{agga14}.

By virtue of \eqref{agga15} and Lemma \ref{agga20}, the operator
$$
W : (\sH_{G_1}, (\cdot,\cdot)_{G_1}) \to  (\sH_{G_2}, (\cdot,\cdot)_{G_2})
$$
defined as $Wf=f$ for $f\in{D}(G_0)$ and extended by continuity onto
$(\sH_{G_1}, (\cdot,\cdot)_{G_1})$ is a unitary mapping between $\sH_{G_1}$ and $\sH_{G_2}$.
Moreover, $WJ_{G_1}=J_{G_2}W$, where
$J_{G_i}$ are the fundamental symmetry operators corresponding to the fundamental decompositions
 (cf. \eqref{e8c})
$\sH_{G_i}=(\hat{\sL}_{+}^{i})^{max}\oplus_{G_{i}}(\hat{\sL}_{-}^{i})^{max}$
of the Krein spaces $(\sH_{G_i}, [\cdot,\cdot]_{G_i})$. Therefore, the indefinite inner products  of these spaces
satisfy the relation
$$
[f,g]_{G_1}=[Wf,Wg]_{G_2}, \qquad f,g\in\sH_{G_1}
$$
and they
are extensions (by the continuity) of the original indefinite inner product $[\cdot,\cdot]$  defined on $D(G_0)$.
For these reasons, the Krein spaces $(\sH_{G_i}, [\cdot,\cdot]_{G_i})$ corresponding to different 
extremal extensions are unitary equivalent and we can identify them.
\end{remark}

Sufficient conditions of quasi maximality are presented below. 

\begin{proposition}\label{agga8}
Dual definite subspaces ${\mathfrak L}_\pm$ are quasi maximal if one of the following (equivalent) conditions is satisfied:
\begin{itemize}
\item[(i)] the operator $G_0$ has a unique nonnegative self-adjoint extension;
\item[(ii)]  dual maximal subspaces ${\mathfrak L}_\pm^{max}\supseteq{\mathfrak L}_\pm$ are determined
by the Friedrichs extension $G_\mu$  of $G_0$
(i.e.,  ${\mathfrak L}_\pm^{max}=(I+T_\mu)\sH_{\pm}$);
\item[(iii)] for all nonzero vectors $g\in\ker(I+G_0^*)=\sH\ominus(M_-\oplus{M_+})$
$$
\inf_{f\in{D}(G_0)}\frac{(G_0f,f)}{|(f,g)|^2}=0;
$$
\item[(iv)] for all nonzero vectors $g\in\sH\ominus{D(T_0)}=\sH\ominus(M_-\oplus{M_+})$
\begin{equation}\label{agga6}
\sup_{x\in{D}(T_0)}\frac{|(T_0x, g)|^2}{\|x\|^2-\|T_0x\|^2}=\infty.
\end{equation}
\end{itemize}
\end{proposition}
\begin{proof}
Let  $G_0$ be a unique nonnegative self-adjoint extension $G$.  Then, $G=G_\mu=G_M$ and, by virtue of \eqref{agga29},
$JG=G^{-1}J$.  Therefore,  the operator $G$  determines dual maximal subspaces  ${\mathfrak L}_\pm^{max}\supseteq{\mathfrak L}_\pm$.
Furthermore, $G$ is an extremal extension (since the Friedrichs extension and the Krein-von Neumann extension are extremal ones).
 In view of Theorem \ref{agga28}, ${\mathfrak L}_\pm$ are quasi maximal. Thus the condition (i) ensures the
 quasi maximality of  ${\mathfrak L}_\pm$.

The condition (ii) is equivalent to (i) due to \eqref{agga29} and \eqref{agga14}.
The equivalence (i) and (iii) follows form \cite[Theorem 9]{AK_Krein}.
The condition  (i)  reformulated for the Cayley transformation $T_0$ of $G_0$  (see \eqref{agga7})
means that $T_0$ has a unique self-adjoint contractive extension $T=T_\mu=T_M$. The latter
is equivalent to (iv) due to \cite[Theorem 6]{AK_Krein}.
\end{proof}

Assume that dual subspaces ${\mathfrak L}_\pm$ do not satisfy conditions of Proposition \ref{agga8}.
Then $T_\mu\not=T_M$ and the subspace  $\mathfrak{M}=\overline{R(T_M-T_\mu)}$ of  $\sH$
is nontrivial.  In view of \eqref{agga4},  this subspace reduces the operator $J$.  Therefore, the restriction of $J$ onto
$\mathfrak{M}$ determines the  operator of fundamental symmetry in $\mathfrak{M}$ and the space $\mathfrak{M}$
endowed with the indefinite inner product $[\cdot,\cdot]$ is the Krein space $(\mathfrak{M}, [\cdot,\cdot])$.

A subspace $\mathfrak{M}_1\subset\mathfrak{M}$ is called \emph{hypermaximal neutral} if the space
 $\mathfrak{M}$ can be decomposed $\mathfrak{M}=\mathfrak{M}_1\oplus{J}\mathfrak{M}_1$.

 Not every Krein space contains hypermaximal neutral subspaces. The sufficient and necessary condition is the coincidence
 of the dimension of a maximal positive subspace with the dimension of a maximal negative one \cite{AK_Azizov}.

\begin{theorem}\label{agga30}
Let dual subspaces ${\mathfrak L}_\pm$ do not satisfy conditions of Proposition  \ref{agga8}.
Then ${\mathfrak L}_\pm$ are quasi maximal subspaces
if and only if  the Krein space $(\mathfrak{M}, [\cdot,\cdot])$ contains a hypermaximal neutral subspace.
\end{theorem}
\begin{proof}  Assume that ${\mathfrak L}_\pm$ are quasi maximal. By Theorem \ref{agga28},
this means the existence of an extremal extension $G\supset{G}_0$ with condition
 \eqref{agga14}. Let $T$ be the Cayley transformation of $G$, see \eqref{agga16}.
By virtue of  Theorem \ref{agga22b}, the operator $T$ is described
by  \eqref{agga22},  where  $X$ is a solution of \eqref{agga23}.
Due  to  \cite[Section 7]{AK_Arlin},  extremal extensions  are specified in \eqref{agga22}
by the assumption that $X$ is an orthogonal projection in $\mathfrak{M}$.
Denote $\mathfrak{M}_1=X\mathfrak{M}$ and $\mathfrak{M}_2=(I-X)\mathfrak{M}$.
Since $X$ is the solution of \eqref{agga23} we decide that   $J\mathfrak{M}_1=\mathfrak{M}_2$.
Therefore, $\mathfrak{M}=\mathfrak{M}_1\oplus{J}\mathfrak{M}_1$ and
 $\mathfrak{M}_1$ is a hypermaximal neutral subspace of the Krein space
 $(\mathfrak{M}, [\cdot,\cdot])$.

 Conversely, let a hypermaximal neutral subspace $\mathfrak{M}_1$ be given.
 Then  $\mathfrak{M}=\mathfrak{M}_1\oplus{J}\mathfrak{M}_1$  and orthogonal projection $X$ on
 $\mathfrak{M}_1$ turns out to be the solution of \eqref{agga23}. The formula \eqref{agga22}
 with given $X$  determines  self-adjoint strong contraction  $T$ anticommuting with $J$ and
 its Cayley transformation $G$ defines the Hilbert space $(\mathfrak{H}_G, (\cdot,\cdot)_G)$ in which
 the direct sum \eqref{e8} is a dense set.
\end{proof}
\begin{corollary}\label{agga31}
 Let the Krein space $(\mathfrak{M}, [\cdot,\cdot])$ contain a hypermaximal neutral subspace. Then
 there exist infinitely many extensions of ${\mathfrak L}_\pm$  to dual maximal subspaces ${\mathfrak L}_\pm^{max}$
 such that $\cD={\mathfrak L}_+[\dot{+}]{\mathfrak L}_-$ is  a dense set in the Hilbert space $(\sH_G, (\cdot,\cdot)_G)$ associated with
 $\cD_{max}={\mathfrak L}_+^{max}[\dot{+}]{\mathfrak L}_-^{max}$  and,  at the same time,
 there exist  infinitely many extensions ${\mathfrak L}_\pm^{max}\supset{\mathfrak L}_\pm$
 such that $\cD$ is not dense in the corresponding Hilbert space $(\sH_G, (\cdot,\cdot)_G)$.
 \end{corollary}
\begin{proof} It follows from the proof of Theorem \ref{agga30} that a
hypermaximal neutral subspace $\mathfrak{M}_1$ determines the dual maximal subspaces
${\mathfrak L}_\pm^{max}\supseteq{\mathfrak L}_\pm$ such that $\cD$ is a dense set in the Hilbert space
 $(\sH_G, (\cdot,\cdot)_G)$ associated with $\cD_{max}$.  The point is that the
orthogonal projection $X$ on $\mathfrak{M}_1$ in  $\mathfrak{M}$ defines the subspaces ${\mathfrak L}_\pm^{max}$.
Therefore, one can construct infinitely many such extensions ${\mathfrak L}_\pm^{max}$  because there are infinitely many
 hypermaximal neutral subspaces in the Krein space (of course, if at least one exists).

If $X$ is the orthogonal projection on $\mathfrak{M}_1$, then $I-X$ is  the orthogonal projection on the hypermaximal neutral subspace
$J\mathfrak{M}_1$.  These two operators define  different pairs of dual maximal subspaces  ${\mathfrak L}_\pm^{max}\supset{\mathfrak L}_\pm$. 
 At the same time,  the operators $X_\alpha=(1-\alpha)X+\alpha(I-X), \ \alpha\in(0,1)$
solve \eqref{agga23} and they determine  dual maximal subspaces ${\mathfrak L}_\pm^{max}(\alpha)\supseteq{\mathfrak L}_\pm$
by the formulas \eqref{agga21} and  \eqref{agga22}.  The linear manifold $\cD$ can not be dense in the Hilbert space
$(\sH_G, (\cdot,\cdot)_G)$  associated with ${\mathfrak L}_+^{max}(\alpha)[\dot{+}]{\mathfrak L}_-^{max}(\alpha)$ 
 because $X_\alpha$ loses the property of being orthogonal projector
($X_\alpha^2\not=X_\alpha$) and $G$ cannot be an extremal extension.
\end{proof}

By virtue of the results above, one can classify dual subspaces  ${\mathfrak L}_\pm$
 in dependence of behavior of the soft $T_M$ and the hard $T_\mu$ extensions of $T_0$.
Precisely:

\begin{itemize}
\item[(A)] \ $T_\mu=T_M$ $\iff$ the subspaces  ${\mathfrak L}_\pm$ are quasi maximal,
there is unique extension of ${\mathfrak L}_\pm$ on to dual maximal subspaces ${\mathfrak L}_\pm^{max}$
$$
\cD={\mathfrak L}_+[\dot{+}]{\mathfrak L}_-  \ \rightarrow \  \cD_{max}={\mathfrak L}_+^{max}[\dot{+}]{\mathfrak L}_-^{max}
$$
and the linear manifold $\cD$ is dense in the Hilbert space $(\sH_G, (\cdot,\cdot)_G)$ associated with $\cD_{max}$;
\item[(B)] $T_\mu\not=T_M$ and
the Krein space $(\mathfrak{M}, [\cdot,\cdot])$ contains hypermaximal neutral subspaces $\iff$  the subspaces ${\mathfrak L}_\pm$ are quasi maximal, there are infinitely many extensions  $\cD \to \cD_{max}$  such that $\cD$ is dense in the Hilbert space $(\sH_G, (\cdot,\cdot)_G)$
associated with $\cD_{max}$ and simultaneously, there are infinitely many extensions  $\cD \to \cD_{max}$ for which $\cD$ cannot be a
dense set in $(\sH_G, (\cdot,\cdot)_G)$;
\item[(C)] $T_\mu\not=T_M$ and the Krein space $(\mathfrak{M}, [\cdot,\cdot])$  does not contain
 hypermaximal neutral subspaces $\iff$  the subspaces ${\mathfrak L}_\pm$ are not quasi maximal,
the total amount of possible extensions  ${\mathfrak L}_\pm \to {\mathfrak L}_\pm^{max}$ is not specified (a unique extension is possible as well as infinitely many ones), the linear manifold $\cD$  is not dense in the Hilbert space $(\sH_G, (\cdot,\cdot)_G)$ associated with $\cD_{max}$.
\end{itemize}

\subsection{Examples.}
Let $\sL_{\pm}$ be dual definite subspaces
and let ${\mathfrak L}_\pm^{max}\supset\sL_{\pm}$  be dual maximal definite subspaces.
The subspaces ${\mathfrak L}_\pm^{max}$ are described by \eqref{agga21},  where $T$ is
a self-adjoint strong contraction anticommuting with $J$. Similarly,
the subspaces ${\mathfrak L}_\pm$ are determined by  \eqref{agga20b}, where $T_0$ is the
restriction of $T$ onto $D(T_0)=M_+\oplus{M_-}$  where  $M_\pm$ are subspaces of  $\sH_\pm$. 
Denote
\begin{equation}\label{neww723}
\Xi=\sqrt{I-T^2}.
\end{equation}
The operator $\Xi$ is a positive self-adjoint contraction in $\sH$ which leaves 
the subspaces $\sH_{\pm}$ invariant.  
\begin{lemma}\label{new23}
The direct sum of  $\sL_{\pm}$ is a dense set in the Hilbert space
$(\sH_G,  (\cdot,\cdot)_G)$ associated with the sum ${\mathfrak L}_+^{max}[\dot{+}]{\mathfrak L}_-^{max}$ of dual maximal definite subspaces
${\mathfrak L}_\pm^{max}\supset\sL_{\pm}$ if and only if 
\begin{equation}\label{new38}
R(\Xi)\cap(\sH\ominus(M_-\oplus{M_+}))=\{0\}.
\end{equation}
\end{lemma}
\begin{proof}
It follows from \eqref{fff7}:
$$
\|f\|_G^2=(Gf, f)=\|\sqrt{G}f\|^2=\|\Xi{x}\|^2, \qquad  f=(I+T)x\in{D}(G).
$$
Therefore, $\{{f}_n\}$ is a Cauchy sequence in $(\sH_G,  (\cdot,\cdot)_G)$
 if and only if  $\{\Xi{x_n}\}$ is a Cauchy sequence in $(\sH, (\cdot,\cdot))$. This means that 
 a one-to-one correspondence between  $\sH_G$ and $\sH$ can be established as follows:
 $$
 {f}_n\to{F_\gamma}\in\sH_G \  (\mbox{wrt.} \ \|\cdot\|_{G}) \iff \Xi{x_n}\to\gamma\in\sH  \  (\mbox{wrt.} \ \|\cdot\|).
 $$

 Let as assume that $F\in\sH_G$ is orthogonal to ${\mathfrak L}_+[\dot{+}]{\mathfrak L}_-$.
 Then $F=F_\gamma$, where $\gamma\in\sH$ and  for all $f\in{\mathfrak L}_+[\dot{+}]{\mathfrak L}_-=(I+T)D(T_0)$
 $$
0=(F_\gamma, f)_G=\lim_{n\to\infty}({f}_n, f)_{G}=\lim_{n\to\infty}(\Xi{x_n}, \Xi{x}_0)=(\Xi\gamma, x_0),
$$
where $x_0$ runs $D(T_0)$. By virtue of \eqref{neww4} and \eqref{new38}, $\gamma=0$. Therefore, 
$F_\gamma=0$.
\end{proof}

\subsubsection{How to construct dual quasi maximal
subspaces?}
We consider below an example (inspired by \cite{AK_AK}) which illustrates a general method of the construction of dual quasi maximal
subspaces.

Let $\{\gamma_n^+\}$ and $\{\gamma_n^-\}$ be orthonormal bases of subspaces $\sH_{\pm}$
in the fundamental decomposition \eqref{AK10}.
Every $\phi\in{\sH}$ has the representation
\begin{equation}\label{fff5}
\phi=\gamma^++\gamma^-=\sum_{n=1}^{\infty}(c_n^+\gamma_n^++c_n^-\gamma_n^-), \quad \gamma^{\pm}=\sum_{n=1}^{\infty}c_n^{\pm}\gamma_n^{\pm}\in\sH_{\pm},
\end{equation}
where $\{c_n^\pm\}\in{l_2(\mathbb{N})}$.
The operator
\begin{equation}\label{fff3}
T\phi=\sum_{n=1}^{\infty}i\alpha_n(c_n^+\gamma_n^--c_n^-\gamma_n^+),\qquad
\alpha_n=1-\frac{1}{n}
\end{equation}
is a self-adjoint contraction anticommuting with the fundamental symmetry $J$
of the Krein space $(\sH, [\cdot,\cdot])$.

The subspaces $\sL_{\pm}^{max}$ defined by \eqref{agga21} with the operator $T$ above  
are dual maximal definite. But they cannot be uniformly definite  since $\|T\|=1$, see Lemma \ref{agga1}.  
 
Let us fix elements $\chi^{\pm}{\in}\sH$,
\begin{equation}\label{agga28c}
\chi^+=\sum_{n=1}^{\infty}\frac{1}{n^\delta}\gamma_n^+,  \quad  \chi^-=\sum_{n=1}^{\infty}\frac{1}{n^\delta}\gamma_n^-,  \qquad
\delta>\frac{1}{2}
\end{equation}
and define the following subspaces of $\sH_\pm$:
$$
M_{+}=\{\gamma^+\in{\sH_+}: (\gamma^+,\chi^+)=0\}, \quad M_{-}=\{
\gamma^-\in{\sH_{-}} : (\gamma^-,\chi^-)=0\}.
$$

\begin{proposition}{\label{newprop}}
The dual definite subspaces 
\begin{equation}\label{agga29}
\sL_+=(I+T)M_+, \qquad \sL_-=(I+T)M_-
\end{equation}
 are quasi maximal for $\frac{1}{2}<\delta\leq\frac{3}{2}$.
In particular, $\frac{1}{2}<\delta\leq{1}$ corresponds to the case (A); the case (B) holds 
when  $1<\delta\leq\frac{3}{2}$.
\end{proposition}
\begin{proof}
First of all we note that $\cD={\mathfrak L}_+[\dot{+}]{\mathfrak L}_-$ is a dense set in $\sH$  for $\frac{1}{2}<\delta\leq\frac{3}{2}$ \cite[p. 317]{AK_AK}.

The subspaces $\sL_\pm$ in \eqref{agga29} are the restriction of dual maximal subspaces ${\mathfrak L}_\pm^{max}=(I+T)\sH$.
Let us show that, for $\frac{1}{2}<\delta\leq{1}$,  the set $\cD$
 is dense in the Hilbert space
$(\sH_G,  (\cdot,\cdot)_G)$ associated with  ${\mathfrak L}_\pm^{max}$.
Due to Lemma \ref{new23}, we should check 
\eqref{new38}. In view of \eqref{fff5}, \eqref{fff3}: 
$$
R(\Xi)=R(\sqrt{I-T^2})=\left\{\sum_{n=1}^{\infty}\sqrt{(1-\alpha_n^2)}(c_n^+\gamma_n^++c_n^-\gamma_n^-) \ : \  \forall\{c_n^\pm\}\in{l_2(\mathbb{N})}\right\}.
$$
On the other hand, $\sH\ominus(M_-\oplus{M_+})=\mbox{span}\{\chi^+, \chi^{-}\}$. Hence, the set   
$R(\Xi)\cap(\sH\ominus(M_-\oplus{M_+}))$ contains nonzero elements if and only if 
$$
\left\{\frac{1}{n^\delta\sqrt{(1-\alpha_n^2)}}=\frac{1}{n^{\delta-1/2}\sqrt{(2-1/n)}}\right\}\in{l_2(\mathbb{N})}
$$
that is impossible for $\frac{1}{2}<\delta\leq{1}$.  Therefore, relation \eqref{new38} holds and
$\sL_\pm$ are quasi maximal subspaces.  By \cite[Proposition 6.3.9]{AK_AK}, 
the dual subspaces $\sL_\pm$ have a unique extension to dual maximal ones ${\mathfrak L}_\pm^{max}$
when $\frac{1}{2}<\delta\leq{1}$ that corresponds to the case $(A)$.

If ${1}<\delta\leq\frac{3}{2}$, then the set $\cD$ cannot be dense in the Hilbert space $(\sH_G,  (\cdot,\cdot)_G)$
considered above.  However, for such values of $\delta$,  the dual subspaces  $\sL_\pm$ can be extended to different pairs of
dual maximal subspaces \cite[Proposition 6.3.9]{AK_AK}. This means that the subspace
$\mathfrak{M}=\overline{R(T_M-T_\mu)}$ is not trivial. The operators $T_\mu, T_M$ are extensions of $T_0$.  
Hence, they coincide on $D(T_0)=M_-\oplus{M_+}$.
Taking into account that $\sH=\mathfrak{M}\oplus\ker(T_M-T_\mu)$, we conclude that
$\mathfrak{M}\subset\mbox{span}\{\chi^+, \chi^{-}\}$. Moreover, $\mathfrak{M}=\mbox{span}\{\chi^+, \chi^{-}\}$. 
Indeed, if $\dim\mathfrak{M}=1$, then the operator $J$ in \eqref{agga23} coincides with $+I$ (or $-I$).  In this case,
the equation \eqref{agga23} has a unique solution $X=\frac{1}{2}I$ and, by Theorem \ref{agga22b},  there exists a unique extension
${\mathfrak L}_\pm^{max}\supset\sL_\pm$ that is impossible. Therefore,  
$\dim\mathfrak{M}=2$ and $\mathfrak{M}=\mbox{span}\{\chi^+, \chi^{-}\}$.
The Krein space $(\mbox{span}\{\chi^+, \chi^{-}\},  [\cdot, \cdot])$ contains infinitely many hypemaximal neutral subspaces.
By Theorem \ref{agga30},  there are infinitely many extensions 
 ${\mathfrak L}_\pm^{max}\supset\sL_\pm$, ${\mathfrak L}_\pm^{max}\not=(I+T)\sH$ such that
 $\cD$ is a dense set in the corresponding Hilbert spaces $(\sH_G,  (\cdot,\cdot)_G)$. Therefore,  $\sL_\pm$
 are quasi maximal (the case  $(B)$).
\end{proof}

\subsubsection{The uniqueness of dual maximal extension ${\mathfrak L}_\pm^{max}\supset{\mathfrak L}_\pm$  does not mean
that ${\mathfrak L}_\pm$ are quasi maximal.}\label{sec4.2.2}
Let us assume that $\chi^+=0$ and $\chi^{-}$ is defined as in \eqref{agga28c}. 
Then $M_+=\sH_+$ and the subspace $\sL_+$  in \eqref{agga29} coincides with 
${\mathfrak L}_+^{max}$.  This means that
the dual maximal definite subspaces ${\mathfrak L}_\pm^{max}\supset{\mathfrak L}_\pm$  are determined uniquely.
Precisely, ${\mathfrak L}_+^{max}=\sL_{+}$  and 
\begin{equation}\label{agga12}
{\mathfrak L}_-^{max}=\sL_{+}^{[\perp]}\supset{\mathfrak L}_-=(I+T)M_-,
\end{equation}
where $\sL_{+}^{[\perp]}$ denotes the maximal negative subspace orthogonal to $\sL_{+}$ with respect to the indefinite inner product
$[\cdot, \cdot]$. 

It follows from Proposition \ref{newprop} that the dual definite subspaces 
${\mathfrak L}_+^{max}$ and $\sL_-$ are quasi maximal (the case (A) of the classification above)
for $\frac{1}{2}<\delta\leq{1}$.

Reasoning by analogy with the proof of Proposition  \ref{newprop} we also conclude that the direct sum 
${\mathfrak L}_+^{max}[\dot{+}]{\mathfrak L}_-$ cannot be dense in the Hilbert space $(\sH_G,  (\cdot,\cdot)_G)$
constructed by ${\mathfrak L}_\pm^{max}$  for $1<\delta\leq\frac{3}{2}$. 
Therefore, the subspaces ${\mathfrak L}_+^{max}$ and $\sL_-$ cannot be quasi maximal  (the case $(C)$ of the classification above).

\section{Operator of $\cC$-symmetry associated with dual maximal definite subspaces}\label{5}
 Let  ${\mathfrak L}_\pm$ be dual definite subspaces such that theirs direct sum is a dense set in  
$\sH$. An operator $\cC_0$ \emph{associated with ${\mathfrak L}_\pm$}
is defined as follows: its domain $D(\cC_0)$ coincides with ${\mathfrak L}_+[\dot{+}]{\mathfrak L}_-$
and
\begin{equation}\label{new5}
\cC_0{f}=f_{+}-f_{-},  \qquad f=f_{+}+f_{-}\in{D(\cC_0)},  \quad f_\pm\in{\mathfrak L}_\pm.
\end{equation}

If $\cC_0$ is given, then the corresponding dual subspaces ${\mathfrak L}_\pm$
are recovered by the formula
$\sL_\pm=\frac{1}{2}(I\pm\cC_0)D(\cC_0)$.

\begin{proposition}\label{neww37}
The following assertions are equivalent:
 \begin{itemize}
  \item[(i)] $\cC_0$ is determined by dual subspaces ${\mathfrak L}_\pm$   with the use
  of  \eqref{new5};
   \item[(ii)] $\cC_0$ satisfies the relation $\cC_0^2f=f$ for all $f\in{D(\cC_0)}$ and
  $J\cC_0$ is a closed densely defined positive symmetric operator in $\sH$.
\end{itemize}
\end{proposition}
\begin{proof}   
By virtue of \eqref{agga15} and \eqref{new5}, $(G_0f, g)=[\cC_0{f},g]$,  $\forall{f, g}\in\cD(G_0)$.
Therefore, $G_0=J\cC_0$, where $G_0$ is a closed densely defined positive symmetric operator acting in $\sH$ and
defined by \eqref{agga7}. The implication $(i)\to(ii)$ is proved.

$(ii)\to(i)$. The operator $G_0=J\cC_0$  satisfies \eqref{AK71} (since $\cC_0^2=I$ on $\cD(G_0)$). 
This operator has the Cayley transform $T_0$ (see \eqref{agga7}) which is a strong contraction in $\sH$
with the condition \eqref{fff1}. Substituting $T_0$ into \eqref{agga20b} we obtain the required dual
definite subspaces  ${\mathfrak L}_\pm$ which generate $\cC_0$.
\end{proof}
 
 We will say that  $\cC_0$ is \emph{an operator of $\cC$-symmetry}  if  $\cC_0$ is
 associated with dual maximal definite subspaces  ${\mathfrak L}_\pm^{max}$. 
 In this case, the notation $\cC$ will be used instead of $\cC_0$.  An operator of
 $\cC$-symmetry admits the presentation $\cC=Je^Q$, where $Q$ is a self-adjoint operator in $\sH$ 
 such that $JQ=-QJ$ \cite{KS}.
 
 Let $\cC_0$ be an operator associated with dual definite subspaces ${\mathfrak L}_\pm$. 
 Its extension to the operator of $\cC$-symmetry $\cC$ is equivalent to the construction of 
 dual maximal definite subspaces ${\mathfrak L}_\pm^{max}\supset\sL_{\pm}$. 
By Theorem \ref{agga2}, each operator $\cC_0$ can be extended to an operator of $\cC$-symmetry
which, in general, is not determined uniquely. Its choice $\cC\supset\cC_0$ determines
the new Hilbert space $(\sH_G,  (\cdot,\cdot)_G)$, where 
\begin{equation}\label{neww11}
G=J\cC=e^Q.
\end{equation}

If ${\mathfrak L}_\pm$ are quasi maximal, then there exists 
a dual maximal extension
${\mathfrak L}_\pm^{max}\supset{\mathfrak L}_\pm$ such that $\cC_0$ is extended to $\cC$
by continuity in the new Hilbert space  $(\sH_G,  (\cdot,\cdot)_G)$ 
 generated by ${\mathfrak L}_\pm^{max}$. 

If the subspaces ${\mathfrak L}_\pm$  are not quasi maximal, then
the sum ${\mathfrak L}_+[\dot{+}]{\mathfrak L}_-$  loses the property of being dense 
in each Hilbert space $(\sH_G, (\cdot,\cdot)_G)$ constructed by the dual maximal subspaces
${\mathfrak L}_\pm^{max}\supset{\mathfrak L}_\pm$. In this case,  the extension by continuity of $\cC_0$ to an operator of $\cC$-symmetry
is impossible. In other words, all possible extensions of $\cC_0$  to an operator of $\cC$-symmetry generate Hilbert spaces
 $(\sH_G, (\cdot,\cdot)_G)$  each of which contains a nontrivial part $\sH_G\ominus_G({\mathfrak L}_+[\dot{+}]{\mathfrak L}_-)$ 
 that has no direct relationship to the original dual subspaces ${\mathfrak L}_\pm$. 

Summing up, we can rephrase the classification in Section \ref{4.1} as follows:
\begin{itemize}
\item[(A{'})]  the extension of $\cC_0$ to an operator of $\cC$-symmetry is unique and it coincides with the extension of 
$\cC_0$ by continuity in the Hilbert space $(\sH_G,  (\cdot,\cdot)_G)$, where $G=G_\mu$ is the Friedrichs extension of $G_0$;
\item[(B')] the extension of $\cC_0$ to an operator of $\cC$-symmetry is not unique.  Infinitely many operators $\cC$
can be realized via the extension of $\cC_0$ by continuity in the corresponding Hilbert spaces $(\sH_G,  (\cdot,\cdot)_G)$.
At the same time, there exist infinitely many extensions $\cC\supset\cC_0$ 
which generate Hilbert spaces $(\sH_G,  (\cdot,\cdot)_G)$  containing the closure of ${\mathfrak L}_+[\dot{+}]{\mathfrak L}_-$ 
as a proper subspace;
 \item[(C')] there is no extension of $\cC_0$ to an operator of $\cC$-symmetry by continuity in the corresponding Hilbert space $(\sH_G,  (\cdot,\cdot)_G)$.
 The total amount of possible extensions  $\cC\supset\cC_0$ is not specified (a unique extension is possible as well as infinitely many ones).
 \end{itemize}
 
\section{$J$-orthonormal quasi bases}\label{6}
\subsection{Quasi bases.}
 A sequence $\{f_n\}$ is called orthonormal in the Krein space $(\sH, [\cdot,\cdot])$, (briefly,  $J$-orthonormal)
if $\{f_n\}$ is orthonormalized with respect to the indefinite inner product $[\cdot,\cdot]$, i.e.,
if  $|[f_n, f_m]|=\delta_{nm}$.  In what follows, the sequence $\{f_n\}$  is assumed to be complete 
(i.e., the span of $\{f_n\}$ is a densely defined set in $\sH$). 

Separating the sequence  $\{f_n\}$ by the signs of $[f_n,f_n]$:
\begin{equation}\label{bebe95}
f_{n}=\left\{\begin{array}{l}
f_{n}^+ \quad \mbox{if} \quad [f_{n},f_{n}]=1, \\
f_{n}^- \quad \mbox{if} \quad [f_{n},f_{n}]=-1
\end{array}\right.
\end{equation}
we obtain two sequences of positive $\{f_n^+\}$ and negative $\{f_n^-\}$  orthonormal elements.
Denote by  ${\mathfrak L}_+$ and ${\mathfrak L}_-$  the closure in the Hilbert space $(\sH, (\cdot, \cdot))$ of
the linear spans generated by the sets $\{f_n^+\}$ and $\{f_n^-\}$, respectively. 
By the construction, ${\mathfrak L}_\pm$ are dual definite subspaces and their direct sum \eqref{e8} is dense in ${\mathfrak H}$.

\begin{definition} A complete $J$-orthonormal sequence $\{f_n\}$ is called a quasi basis
of the Hilbert space $(\sH, (\cdot, \cdot))$
if the corresponding dual subspaces ${\mathfrak L}_\pm$ are quasi maximal.
\end{definition}

If a $J$-orthonormal sequence $\{f_n\}$ is a basis in $\sH$, then
the dual subspaces ${\mathfrak L}_\pm$
are maximal definite, i.e., ${\mathfrak L}_\pm={\mathfrak L}_\pm^{max}$  \cite[Statement 10.12 in Chapter 1]{AK_Azizov}.
Therefore,  $\{f_n\}$ is also a quasi basis in $\sH$. 

\begin{proposition}
Let a $J$-orthonormal sequence $\{f_n\}$ be a quasi basis in $\sH$, then
its biorthogonal sequence $\{\gamma_n=\mbox{sign}([f_n,f_n])Jf_n\}$ is also a quasi basis.
\end{proposition}
\begin{proof}  The  biorthogonal sequence $\{\gamma_n\}$ determines dual definite subspaces $J{\mathfrak L}_\pm$.
 The corresponding positive symmetric operator associated with $J{\mathfrak L}_+[\dot{+}]J{\mathfrak L}_-$ (see \eqref{agga7})             
 coincides with $G_0^{-1}$. By Theorem \ref{agga28} there exists an extremal extension $G\supset{G_0}$ which satisfies \eqref{agga14}.
 Then, $G^{-1}$ is an extremal extension of $G_0^{-1}$ which satisfies   \eqref{agga14}. Therefore, $J{\mathfrak L}_\pm$
 are quasi maximal.
 \end{proof}

\begin{theorem}\label{agga38} 
Let  $\{f_n\}$ be a complete $J$-orthonormal sequence in $\sH$ and
let $\cC_0$ be associated with the subspaces ${\mathfrak L}_\pm$ generated by $\{f_n\}$
The following statements are equivalent:
\begin{itemize}
\item[(i)]	 $\{f_n\}$ is a quasi basis of $\sH$;
\item[(ii)] there exists an operator of $\cC$-symmetry $\cC\supset\cC_0$
such that $\{f_n\}$ turns out to be an orthonormal basis in the
Hilbert space  $(\sH_G, (\cdot, \cdot)_G)$ generated by $\cC$.
\item[(iii)]  there exists a self-adjoint operator $Q$ in the Hilbert space $\sH$, which anticommutes with
$J$ and such that the sequence $\{g_n=e^{Q/2}f_n\}$ is an orthonormal basis of $\sH$ and each $g_n$ belongs to one of the subspaces
$\sH_\pm$ of the fundamental decomposition \eqref{AK10}.
\end{itemize}
\end{theorem}
\begin{proof} $(i)\to(ii)$. If  $\{f_n\}$ is a quasi basis,  then there exists an extremal extension $G\supset{G_0}$ which  satisfies \eqref{agga14}
and the linear span of $\{f_n\}$ is dense in $(\sH_G, (\cdot, \cdot)_G)$.
It follows from \eqref{agga15} and \eqref{neww11} that  
\begin{equation}\label{neww49}
(f_n, f_m)_G=(Gf_n, f_m)=[\cC{f}_n, f_m]=\delta_{nm}.
\end{equation}
Therefore, $\{f_n\}$ is an orthonormal basis in the Hilbert space $(\sH_G, (\cdot, \cdot)_G)$.
The inverse implication $(ii)\to(i)$ is obvious. 

$(ii)\to(iii)$. Since $\cC=Je^Q$ and $G=e^Q$,  where $Q$ satisfies the condition of item $(iii)$,
the relation \eqref{neww49}  takes the form  $(e^{Q/2}f_n, e^{Q/2}f_m)=\delta_{nm}$.
Hence, $\{g_n=e^{Q/2}f_n\}$ is an orthonormal sequence in $\sH$.  The completeness of 
$\{g_n\}$ in $\sH$ will be established with the use of Lemma \ref{new23}.  Before doing this we
note that  the dual maximal definite subspaces corresponding to 
$\cC=Je^Q$ are given by \eqref{agga21},  where $T=-\tanh\frac{Q}{2}$ \cite{KS}. 
Therefore, the bounded operator $\Xi$ in \eqref{neww723}  coincides with $\cosh^{-1}{Q/2}$,  where
 $\cosh{Q/2}=\frac{1}{2}(e^{Q/2}+e^{-Q/2})$.    
 
Each $g_n$ belongs to the domain of definition of $\cosh{Q/2}$ and
\begin{equation}\label{agga45}
(I-\tanh{Q}/{2})\cosh{Q}/{2}g_n=(\cosh{Q}/{2} - \sinh{Q}/{2})g_n=e^{-{Q}/{2}}g_n=f_n.
\end{equation}
Comparing the obtained relation with \eqref{agga20b}, \eqref{neww4} and taking into account 
the definition of ${\mathfrak L}_\pm$,  we conclude that 
$M_-\oplus{M_+}$ coincides with the closure of  $\mbox{span}\{\cosh{Q}/{2}g_n\}$. Therefore,
$$
\sH\ominus(M_-\oplus{M_+})=\sH\ominus\mbox{span}\{\cosh{Q}/{2}g_n\}.
$$ 

Let $u\in\sH$ be orthogonal to $\{g_n\}$. Then 
$$
0=(u, g_n)=(\cosh^{-1}{Q/2}u, \cosh{Q}/{2}g_n)
$$ and hence, $\cosh^{-1}{Q/2}u\in{R}(\Xi)\cap(\sH\ominus(M_-\oplus{M_+}))$.
By virtue of Lemma \ref{new23},  $\cosh^{-1}{Q/2}u=0$. This means that $u=0$ and 
$\{g_n\}$ is a complete orthonormal sequence in $\sH$, i.e., $\{g_n\}$ is a basis in $\sH$.

It follows from \eqref{agga45}, that $\cosh{Q}/{2}g_n$ belongs to $\sH_+$ or $\sH_-$
(depending on either $f_n\in{\mathfrak L}_+$ or $f_n\in{\mathfrak L}_-$).  The same property
holds true for $g_n$ since the operator $\cosh{Q/2}=\frac{1}{2}(e^{Q/2}+e^{-Q/2})$  commutes with $J$.

$(iii)\to(ii)$.  Since $g_n\in\sH_\pm$, we get $Jg_n=\pm{g_n}=Je^{Q/2}f_n=e^{-Q/2}Jf_n$. 
Therefore $g_n\in{D(e^{Q/2})}=R(e^{-Q/2})$.
This means that the sequence $\{\cosh{Q}/{2}g_n\}$ is well defined and  $f_n\in{D}(e^Q)$. 

For given $Q$ we define the operator of $\cC$-symmetry $\cC=Je^Q$ and  set $G=e^Q$. 
By analogy with \eqref{neww49},
$$
\delta_{nm}=(g_n, g_m)=(e^{Q/2}f_n, e^{Q/2}f_m)=(Gf_n, f_m)=(f_n, f_m)_G.
$$
Therefore, $\{f_n\}$ is an orthonormal sequence in $(\sH_G, (\cdot, \cdot)_G)$. 
Furthermore, the relations above mean that $Gf_n=\gamma_n=sign([f_n,f_n])Jf_n$ and
$\cC{f}_n=\cC_0f_n=sign([f_n,f_n])f_n$. Hence, $\cC$ is an extension of $\cC_0$ and
the dual maximal definite subspaces ${\mathfrak L}_\pm^{max}$ determined by $\cC$
are the extensions of the dual definite subspaces  ${\mathfrak L}_\pm$ generated as the
closures of $\{f_n^{\pm}\}$. This fact and \eqref{agga45} lead to the conclusion that 
 ${\mathfrak L}_\pm$ are determined by \eqref{agga20b}, where 
 $D(T_0)=M_-\oplus{M_+}$ coincides with the closure of  $\mbox{span}\{\cosh{Q}/{2}g_n\}$.
 
 Assume that $\{f_n\}$ is not complete in $(\sH_G, (\cdot, \cdot)_G)$.  Then the direct sum
 of ${\mathfrak L}_\pm$ cannot be dense in  $\sH_G$ and, by Lemma \ref{new23} (since $\Xi=\cosh^{-1}{Q/2}$) there exists
 nonzero $p=\cosh^{-1}{Q/2}u$ such that for all $g_n$
 $$
 0=(p, \cosh{Q}/{2}g_n)=(\cosh^{-1}{Q/2}u, \cosh{Q/2}g_n)=(u, g_n)=0
 $$
 that is impossible (since $\{g_n\}$ is a basis of $\sH$). The obtained contradiction means
 that $\{f_n\}$  is an orthonormal basis of $\sH_G$.
\end{proof}

\begin{corollary}\label{agga35}
Let  $\{f_n\}$ be a quasi basis of  $\sH$.
Then there exists an operator of $\cC$-symmetry $\cC=Je^{Q}\supset\cC_0$
such that  for elements of the energetic linear manifold $g\in\mathfrak{D}[G]\subset\sH$:
\begin{equation}\label{agga81}
 g=\sum_{n=1}^\infty[g, {\cC}f_n]f_n,  \qquad e^{Q/2}g=\sum_{n=1}^\infty[g, {\cC}f_n]e^{Q/2}f_n   
\end{equation}
where the series converge in the Hilbert spaces $(\sH_G, (\cdot, \cdot)_G)$ and $(\sH, (\cdot, \cdot))$,
respectively.
\end{corollary}
\begin{proof} 
The energetic linear manifold $\mathfrak{D}[G]$
coincides with the common part of $\sH$ and $\sH_G$ (see Section \ref{ref1}). 
 Hence, each $g\in\mathfrak{D}[G]$
can be presented as $g=\sum_{n=1}^\infty{c_n}f_n$,  where the series converges 
 $\mathfrak{H}_{G}$ and $c_n=(g, f_n)_G=(e^{Q/2}g, e^{Q/2}f_n)=(g, e^{Q}f_n)=[g, {\cC}f_n].$ 
 Similarly, each $e^{Q/2}g$ admits the decomposition $e^{Q/2}g=\sum_{n=1}^\infty{c_n'}e^{Q/2}f_n$,
 where the series converges in $\sH$ and $c_n'=(e^{Q/2}g, e^{Q/2}f_n)=[g, {\cC}f_n].$
 \end{proof}

\begin{corollary}
Let $\{f_n\}$ be a quasi basis of $\sH$. 
Then an operator  of $\cC$-symmetry $\cC$ appearing in items $(ii), (iii)$ of Theorem \ref{agga38} acts 
as follows: 
\begin{equation}\label{new8b}
\cC{f}=\sum_{n=1}^\infty[f, f_n]{f_n}, \qquad \cC{f}=\sum_{n=1}^\infty[e^{Q/2}f, f_n]{e^{Q/2}f_n}, \quad  \forall{f}\in{D}(\cC),
\end{equation}
where the series are convergent in $\sH_G$ and $\sH$, respectively.
 \end{corollary}
\begin{proof}
The  relations in \eqref{new8b} follow from the corresponding formulas in \eqref{agga81} where
$g=\cC{f}$ and $g=e^{-Q/2}\cC{f}$, respectively.  
\end{proof}

An operator $H$ in a Krein space $(\sH, [\cdot, \cdot])$ is called \emph{$J$-symmetric} if
$[Hf, g]=[f, Hg]$ for all $f,g\in{D}(H)$.

\begin{corollary}\label{agga72}
If eigenfunctions $\{f_n\}$ of a $J$-symmetric operator $H$ form a quasi basis in $\sH$,  then
there exists an operator of $\cC$-symmetry $\cC$ such that the operator
 $H$ restricted on $\mbox{span}\{f_n\}$ turns out to be essentially self-adjoint in the Hilbert space $(\sH_G, (\cdot,\cdot)_{G})$
 generated by $\cC$. 
\end{corollary}
\begin{proof} Due to Theorem \ref{agga38} there exists an operator $\cC$ such that $\{f_n\}$ is a basis of $(\sH_G, (\cdot,\cdot)_{G})$. 
 The restriction of $\cC$ on $\mbox{span}\{f_n\}$ coincides with the operator $\cC_0$ defined by \eqref{new5} (here, of course,
 ${\mathfrak L}_\pm$ are the closures of $\{f_n^{\pm}\}$).  It is easy to see that
 $$
 \cC{H}f=\cC_0{H}f=H\cC_0{f}=H\cC{f},  \qquad f\in\mbox{span}\{f_n\}.
 $$
 Taking \eqref{agga15},  \eqref{fff7} into account,  we obtain 
 $$
 (Hf, g)_G=(GHf, g)=[\cC{H}f, g]=[f, \cC{H}g]=(f, GHg)=(f, Hg)_G   
 $$
 for all $f, g\in\mbox{span}\{f_n\}$.
 Hence $H$ is symmetric in $(\sH_G, (\cdot,\cdot)_{G})$. Since $R(H\pm{i}I)\supset\mbox{span}\{f_n\}$,  the operator
 $H$ is essentially self-adjoint in $\sH_G$.
 \end{proof}

\subsection{Examples.}   Examples of quasi-bases can be easy constructed with the use 
of Theorem \ref{agga38}.   Indeed let us assume that  $g_n$ be an orthonormal basis of $(\sH, (\cdot,\cdot))$
such that each $g_n$ belongs to one of the subspaces $\sH_\pm$ of the fundamental decomposition \eqref{AK10}.
Let $Q$ be a self-adjoint operator in $\sH$, which anticommutes with $J$. If all $g_n$ belong to the domain of definition of
$e^{-Q/2}$  then $f_n=e^{-Q/2}g_n$ is an $J$-orthonormal system of the Krein space $(\sH, [\cdot,\cdot])$.
Assuming additionally that $\{f_n\}$ is complete in $\sH$, we get  an example of quasi basis.

\vspace{3mm}

{\bf I.}
Let $\sH=L_2(\mathbb{R})$ and let $J=\mathcal{P}$ be the space parity operator $\mathcal{P}f(x)=f(-x)$.
The subspaces $\sH_{\pm}$ of the fundamental decomposition \eqref{AK10}  coincide with the subspaces of even and odd
functions of  $L_2(\mathbb{R})$. 

The Hermite functions
$$
g_n(x)=\frac{1}{\sqrt{2^nn!\sqrt{\pi}}}H_n(x)e^{-x^2/2},  \quad H_n(x)=e^{x^2/2}(x-\frac{d}{dx})^ne^{-x^2/2}
$$
is an example of orthonormal basis  of $L_2(\mathbb{R})$. The functions $g_n$ are either odd or even functions. 
Therefore, $g_n\in\sH_+$ or $g_n\in\sH_-$. 

Since Hermitian functions are entire functions,  the complex shift of $g_n$ can be defined:
$$
 f_n(x)=g_n(x+ia),  \qquad  a\in\mathbb{R}\setminus\{0\}, \quad n=0,1,2,\ldots
$$ 
The sequence $\{f_n\}$ is complete in $L_2(\mathbb{R})$ \cite[Lemma 2.5]{Mit}.
Applying the Fourier transform
$Ff=\frac{1}{\sqrt{2\pi}}\int_{-\infty}^\infty{e^{-ix\xi}}f(x)dx$
to $f_n$ we get $Ff_n=e^{-a\xi}Fg_n$. Therefore,  $f_n=F^{-1}e^{-a\xi}Fg_n$.
The last relation can be rewritten as
$$
f_n=e^{-Q/2}g_n,   \qquad   Q=-2ai\frac{d}{dx}. 
$$   
This means that  $\{f_n\}$ is a quasi basis of $L_2(\mathbb{R})$. 
The functions  $\{f_n\}$ are simple eigenfunctions of the $\mathcal{P}$-symmetric
operator 
$$
H=-\frac{d^2}{dx^2}+x^2+2iax, \qquad Hf_n=(1+2n+a^2)f_n.
$$
Therefore, $H$ restricted on $\mbox{span}\{f_n\}$ is essentially self-adjoint 
in the new Hilbert space $(\sH_G, (\cdot,\cdot)_G)$, where  $\sH_G$ is the 
completion of $\mbox{span}\{f_n\}$ with respect to the norm:
$\|f\|^2_G=(e^Qf,f)=(F^{-1}e^{2a\xi}Ff, f)$.

{\bf II.} Let $\{g_n\}$ be orthonormal basis in $L_2(\mathbb{R})$
which consists of the eigenfunctions  of the anharmonic oscillator 
$$
H_0=-\frac{d^2}{dx^2} + |x|^\beta, \qquad  \beta>2
$$ 
The eigenfunctions $g_n$ are either even or odd functions. 

Consider the sequence $f_n(x)=e^{p(x)}g_n(x)$, 
where $p\in{C^2}(\mathbb{R})$ is a real valued odd function such that
$$
|p^{k}(x)|\leq{C}(1+x^2)^{\frac{\alpha-k}{2}}, \quad k=0,1,2, \quad \alpha<\beta/2+1.
$$
The sequence $\{f_n\}$ is complete in $L_2(\mathbb{R})$  \cite[Lemma 3.6]{Mit} and
$f_n$ are simple eigenfunctions of the $\mathcal{P}$-symmetric operator
$$
H=H_0+p''(x)- (p'(x))^2+2ip'(x)\frac{d}{dx}.
$$
The sequence $\{f_n\}$ is a quasi basis in $L_2(\mathbb{R})$
(since $f_n=e^{-Q/2}g_n$ with $Q=-2p(x)I$) and 
$H$ restricted on $\mbox{span}\{f_n\}$ is essentially self-adjoint 
in the new Hilbert space $(\sH_G, (\cdot,\cdot)_G)$, where  $\sH_G$ is the 
completion of $\mbox{span}\{f_n\}$ with respect to the norm:
$\|f\|^2_G=(e^Qf,f)=(e^{-2p(x)}f, f)$.

\vspace{3mm}

{\bf III.} \emph{An example of a complete $J$-orthonormal sequence $\{f_n\}$ which cannot be a quasi-basis.}
The dual definite subspaces ${\mathfrak L}_+^{max}$  and  $\sL_-$ considered in Sect. 
\ref{sec4.2.2}  cannot be dual quasi maximal for  $1<\delta\leq\frac{3}{2}$. Therefore,  
each $J$-orthonormal sequence $\{f_n\}$ such that the closure of its positive/negative elements coincide 
with ${\mathfrak L}_+^{max}$ and $\sL_-$,  respectively cannot be quasi basis.

 \end{document}